\documentclass[a4paper]{amsart}

\usepackage{amsthm}
\usepackage{cite}
\usepackage[dvipdfmx]{graphicx}

\theoremstyle{definition}
\newtheorem{theorem}{Theorem}[section]
\newtheorem{lemma}[theorem]{Lemma}
\newtheorem{proposition}[theorem]{Proposition}
\newtheorem{observation}[theorem]{Observation}
\newtheorem{fact}[theorem]{Fact}
\newtheorem{corollary}[theorem]{Corollary}

\theoremstyle{definition}
\newtheorem{definition}[theorem]{Definition}

\newtheorem{pclaim}[theorem]{Claim}

\theoremstyle{remark}
\newtheorem{remark}[theorem]{Remark}

\newenvironment{rmenum}{
\begin{enumerate}

}
{\end{enumerate}}

\newcommand{\gpart}[1]{\mathcal{P}(#1)}
\newcommand{\pargpart}[2]{\mathcal{P}_{#1}(#2)}

\newcommand{\parvcoup}[2]{{}^{\top}U_{#1}(#2)}
\newcommand{\dmcomp}[1]{\mathcal{D}(#1)}
\newcommand{\cdmcomp}[1]{\mathcal{D}^+(#1)}
\newcommand{\idmcomp}[1]{\mathcal{D}^-(#1)}
\newcommand{\admcomp}[2]{\mathcal{D}^{#2}(#1)}

\newcommand{\dmo}[1]{\leq^{\circ}_{#1}} 
\newcommand{\dm}[1]{\leq_{#1}} 
\newcommand{\gdm}{\preceq} 
\newcommand{\gdmo}{\preceq^{\circ}} 
\newcommand{\nono}{\smile^{\circ}} 
\newcommand{\non}{\smile} 
\newcommand{\gdmposet}[1]{\mathcal{O}(#1)}

\newcommand{\parNei}[2]{N_{#1}(#2)}
\newcommand{\parcut}[2]{\delta_{#1}(#2)}

\newcommand{\yield}{\triangleleft}
\newcommand{\ecomp}[1]{\mathcal{G}(#1)}
\newcommand{\const}[1]{\mathcal{G}^+(#1)}
\newcommand{\inconst}[1]{\mathcal{G}^-(#1)}
\newcommand{\ainconst}[2]{\mathcal{G}^-_{#2}(#1)}

\newcommand{\gsim}[1]{\sim_{#1}} 
\newcommand{\defi}[1]{\mathrm{def}(#1)}
\newcommand{\parup}[2]{\mathcal{U}_{#1}(#2)}
\newcommand{\parupstar}[2]{\mathcal{U}^{*}_{#1}(#2)}
\newcommand{\vparup}[2]{U_{#1}(#2)}
\newcommand{\vparupstar}[2]{U^*_{#1}(#2)}

\newcommand{\ds}[1]{K(#1)}
\newcommand{\sd}[1]{\pi(#1)}

\title[]{Nonbipartite Dulmage-Mendelsohn Decomposition for Berge Duality}
\author{Nanao Kita}
\address{National Institute of Informatics
2-1-2 Hitotsubashi, Chiyoda-ku, Tokyo, Japan 101-8430}
\email{kita@nii.ac.jp}
\date{\today}

\begin{document}

\begin{abstract} 
The {\em Dulmage-Mendelsohn decomposition} is a classical {\em canonical decomposition} in matching theory applicable for bipartite graphs, and is famous not only for its application in the field of matrix computation, but also for providing a prototypal structure in  matroidal optimization theory. 
The Dulmage-Mendelsohn decomposition is stated and proved using the two color classes, and therefore generalizing this decomposition for nonbipartite graphs has been a difficult task. 
In this paper, we obtain a new canonical decomposition that is a generalization of the Dulmage-Mendelsohn decomposition for arbitrary graphs, using a recently introduced tool in matching theory, the {\em basilica decomposition}. 
Our result enables us to understand all known canonical decompositions in a unified way. 
Furthermore, we apply our result to derive a new theorem regarding {\em barriers}. 
The duality theorem for the  maximum matching problem is the celebrated {\em Berge formula}, in which dual optimizers are known as barriers. 
Several results regarding maximal barriers have been derived 
by known canonical decompositions, 
however no characterization has been known for general graphs. 
In this paper, we provide a characterization of the family of {\em maximal barriers} in general graphs, 
in which the known results are developed and unified. 

\end{abstract} 

\maketitle

\section{Introduction} 

We establish the Dulmage-Mendelsohn decomposition for general graphs. 
The  {\em Dulmage-Mendelsohn decomposition}~\cite{dm1958, dm1959, dm1963},  
or the {\em DM decomposition} in short,  is a classical canonical decomposition in matching theory~\cite{lovasz2009matching} 
 applicable for bipartite graphs. 
This decompositions is famous for its application for combinatorial matrix theory, especially, for providing an  efficient solution for a system of linear equations~\cite{dm1963, duff1986direct}, 
and is also  important in matroidal optimization theory.

{\em Canonical decompositions} of a graph are fundamental tools in matching theory~\cite{lovasz2009matching}. 
A canonical decomposition partitions a given graph 
in a way uniquely determined for the graph, 
and describes the structure of maximum matchings  using this partition. 
The classical canonical decompositions  
are the {\em Gallai-Edmonds}~\cite{gallai1964, edmonds1965} and {\em Kotzig-Lov\'asz} decompositions~\cite{kotzig1959a, kotzig1959b, kotzig1960}, 
in addition to the DM decomposition. 
The DM and Kotzig-Lov\'asz decompositions 
are applicable for bipartite graphs and {\em factor-connected graphs}, respectively. 
The Gallai-Edmonds decomposition partitions an arbitrary graph 
into three parts,  that is, the so-called $D(G)$, $A(G)$, and $C(G)$ parts. 
Comparably recently, a new canonical decomposition was proposed: 
the {\em basilica decomposition}~\cite{DBLP:conf/isaac/Kita12, DBLP:journals/corr/abs-1205-3816}. This decomposition is applicable for arbitrary graphs 
and contains a generalization of the Kotzig-Lov\'asz decomposition 
and a refinement the Gallai-Edmonds decomposition. 
(The $C(G)$ part can be  decomposed nontrivially.)

In this paper, 
using the basilica decomposition, 
we establish an analogue of the DM decomposition 
for general graphs. 
Our result accordingly provides a paradigm 
that enables us to handle any graph 
and understand the known canonical decompositions in a unified way. 
In the theory of original DM decomposition, 
the concept of the {\em DM components} of a bipartite graph 
is first defined, 
and then it is proved that these components form a poset 
with respect to a certain binary relation. 

This theory is heavily depend on the two color classes of a bipartite graph 
and cannot be easily generalized for nonbipartite graphs. 
In our generalization, 
we first define a generalization of the DM components 
using the basilica decomposition. 
To capture the structure formed by these components in nonbipartite graphs, 
we introduce a little more complexed concept:  {\em posets with a transitive forbidden relation}. 
We then prove that the generalized DM components 
form a poset with a transitive forbidden relation 
 for certain binary relations.

Using this generalized DM decomposition, 
we derive a characterization of the family of {\em maximal barriers}  
in general graphs. 
The {\em Berge formula} is a combinatorial min-max theorem, 
 in which maximum matchings are the optimizers of one hand, 
 and the optimizers of the other hand are known as {\em barriers}~\cite{lovasz2009matching}. 
 That is, barriers are the dual optimizers of the maximum matchings problem. 
Barriers are heavily employed as a tool 
for studying matchings. 
However, not so much is known about barriers themselves~\cite{lovasz2009matching}. 
Aside from several observations that are derived rather easily from the Berge formula, 
several substantial results are known about (inclusion-wise) maximal barriers, 
which are provided by canonical decompositions.

Our result for maximal barriers 
shows a reasonable consistency 
regarding our generalization of the DM decomposition, 
considering the relationship between 
each known canonical decomposition and maximal barriers. 
Each known canonical decomposition 
can be used to state the structure of maximal barriers. 
The original DM decomposition 
provides a characterization of the family of maximal barriers in bipartite graphs 
 in terms of ideals in the poset; 
 minimum vertex covers in  bipartite graphs are equivalent to maximal barriers. 
The Gallai-Edmonds decomposition 
derives a characterization of the intersection of all maximal barriers 
(that is, the $A(G)$ part)~\cite{lovasz2009matching}; 
this characterization is known as the {\em Gallai-Edmonds description}. 
The Kotzig-Lov\'asz decomposition 
is used for characterizing  the family of maximal barriers 
in factor-connected graphs~\cite{lovasz2009matching}; 
this result is known as {\em Lov\'asz's canonical partition theorem}~\cite{lovasz1972structure, lovasz2009matching}. 
The basilica decomposition provides 
 the structure of a given maximal barrier in general graphs, 
which contains a common generalization of 
the Gallai-Edmonds description and Lov\'asz's canonical partition theorem. 
Hence, a generalization of the DM decomposition 
would be reasonable if it can characterize the family of maximal barriers, 
and our generalization attains this in a way analogical to the classical DM decomposition, 
that is, in terms of  ideals in the poset with a transitive forbidden relation.

Our results may imply a new possibility in matroidal optimization theory. 
In submodular function theory, the bipartite maximum matching problem is an important exemplary problem, 
and the DM decomposition therefore has a special role in this theory. 
Our  nonbipartite DM decomposition 
may be a clue to a new phase of  submodular function theory 
that can be brought in by capturing these concepts.

The remainder of this paper is organized as follows: 
In Section~\ref{sec:notation}, we explain the basic definitions. 
In Section~\ref{sec:basilica}, we present the preliminary results 
from the basilica decomposition theory. 
In Section~\ref{sec:ptfp}, 
we introduce the new concept of   posets with a transitive forbidden relation. 
In Section~\ref{sec:gdm}, 
we provide our main result, the nonbipartie DM decomposition. 
In Section~\ref{sec:barrierpre}, 
we present preliminary definitions and results regarding barriers. 
We then prove in Section~\ref{sec:barrier}  that our generalization of the DM decomposition 
can be used to characterize the family of maximal barriers. 
In Section~\ref{sec:original}, 
we show how our result contains the original DM decomposition 
for bipartite graphs. 
In Section~\ref{sec:alg}, we remark that our nonbipartite DM decomposition 
can be computed in polynomial time.

\section{Notation} \label{sec:notation} 
\subsection{General Definitions} 
For basic notation for sets, graphs, and algorithms, 
we mostly follow Schrijver~\cite{schrijver2003}. 
In this section, we explain exceptions or nonstandard definitions. 
In Section~\ref{sec:notation}, unless otherwise stated, 
let $G$ be a graph. 
The vertex set and the edge set of $G$ are denoted by $V(G)$ and $E(G)$, respectively. 
We treat paths and circuits as graphs. 
For a path $P$ and vertices $x$ and $y$ from $P$, 
$xPy$ denotes the subpath of $P$ between $x$ and $y$. 
The singleton set $\{x\}$ is often denoted by just $x$. 
We often treat a graph  as the set of its vertices.

\subsection{Graph Operations} 
In the remainder of this section, let $X\subseteq V(G)$. 
The subgraph of $G$ induced by $X$ is denoted by $G[X]$. 
The graph $G[V(G)\setminus X]$ is denoted by $G-X$. 
The contraction of $G$ by $X$ is denoted by $G/X$. 
Let $F\subseteq E(G)$. 
The graph obtained by deleting $F$ from $G$  
without removing vertices is denoted by $G - F$. 
Let $H$ be a subgraph of $G$. 
The graph obtained by adding $F$ to $H$ is denoted by $H + F$. 
Regarding these operations, 
we identify vertices, edges, subgraphs of the newly created graph 
with  the naturally corresponding items of old graphs.

\subsection{Functions on Graphs} 
A {\em neighbor} of $X$ is a vertex from $V(G)\setminus X$ that is adjacent to some vertex from $X$. 
The neighbor set of $X$ is denoted by $\parNei{G}{X}$. 
Let $Y\subseteq V(G)$. 
The set of edges joining $X$ and $Y$ is denoted by $E_G[X, Y]$. 
The set $E_G[X, V(G)\setminus X]$ is denoted by $\parcut{G}{X}$. 

\subsection{Matchings} 
A set $M \subseteq E(G)$  is a {\em matching} 
if $|\parcut{G}{v} \cap M | \le 1$ holds for each $v\in V(G)$. 
For a matching $M$, we say that $M$ {\em covers} a vertex $v$ if $|\parcut{G}{v} \cap M | = 1$; 
otherwise, we say that $M$ {\em exposes} $v$. 
A matching is {\em maximum} if it consists of the maximum number of edges. 
A graph can possess an exponentially large number of matchings. 
A matching is {\em perfect} 
 if it covers every vertex. 
A graph is {\em factorizable} if it has a perfect matchings. 
A graph is {\em factor-critical}  
if, for each vertex $v$, $G-v$ is factorizable. 
A graph with only one vertex is defined to be factor-critical. 
The number of edges in  a maximum matching is denoted by $\nu(G)$. 
The number of vertices exposed by a maximum matching 
is denoted by $\defi{G}$; that is, $\defi{G} := |V(G)| -2\nu(G)$. 

\subsection{Alternating Paths and Circuits} 

Let $M\subseteq E(G)$. 
A circuit or path is said to be {\em $M$-alternating} 
if edges in $M$ and not in $M$ appear alternately. 
The precise definition is the following: 
A circuit $C$ of $G$ is $M$-alternating if $E(C)\cap M$ 
is a perfect matching of $C$. 
We define the three types of $M$-alternating paths. 
Let $P$ be a path with ends $s$ and $t$. 
We say that $P$ is {\em $M$-forwarding} from $s$ to $t$ 
 if $M\cap E(P)$ is a matching of $P$ 
 that covers every vertex except for $t$. 
We say that $P$ is {\em $M$-saturated } between $s$ and $t$ 
 if $M\cap E(P)$ is a perfect matching of $P$. 
We say that $P$ is {\em  $M$-exposed } between $s$ and $t$ 
 if $M\cap E(P)$ is a matching of $P$ that covers every vertex except for $s$ and $t$. 
Any path with exactly one vertex $x$ is defined to be an $M$-forwarding path from $x$ to $x$, 
and is never treated as an $M$-exposed path. 
Any $M$-forwarding path has an even number of edges, which can be zero, 
whereas any $M$-saturated or -exposed path has an odd number of edges.

A path $P$ is an {\em ear} relative to $X$ 
if the internal vertices of $P$ are disjoint from $X$, 
whereas the ends are in $X$. 
A circuit $C$ is an {\em ear} relative to $X$  
if exactly one vertex of $C$ is in $X$; 
for simplicity, we call the vertex in $X\cap V(C)$ the {\em end } of the ear $C$. 
We call an ear $P$ relative to $X$ an {\em $M$-ear} 
if $P - X$ is empty or an $M$-saturated path, and $\parcut{P}{X} \cap M = \emptyset$.

\subsection{Berge Formula and Barriers} 

We now explain the Berge Formula and the definition of barriers. 
An {\em odd component} (resp. {\em even component}) of a graph is a connected component with an odd (resp. even) number of vertices. 
The number of odd components of $G - X$ is denoted by $q_G(X)$. 
The set of vertices from odd components (resp. even components) of $G-X$ is denoted by $D_X$ 
(resp. $C_X$). 

\begin{theorem}[Berge Formula~\cite{lovasz2009matching}]
For a graph $G$, 
$\defi{G}$ is equal to the maximum value of $q_G(X) - |X|$, 
where $X$ is taken over all subsets of $V(G)$. 
\end{theorem} 

The set of vertices that attains the maximum value in this relation 
is called a {\em barrier}. 
That is, a set of vertices $X$ is a {\em barrier} 
if $\defi{G} = q_G(X) - |X|$.

\subsection{Gallai-Edmonds Family and Structure Theorem}

The set of vertices that can be exposed by some maximum matchings 
is denoted by $D(G)$. 
The neighbor set of $D(G)$ is denoted by $A(G)$, 
and the set $V(G)\setminus D(G) \setminus A(G)$ is denoted by $C(G)$. 
The following statement about $D(G)$, $A(G)$, and $C(G)$ 
is the celebrated {\em Gallai-Edmonds structure theorem}~\cite{lovasz2009matching, gallai1964, edmonds1965}.

\begin{theorem}[Gallai-Edmonds Structure Theorem] \label{thm:ge}
For any graph $G$,  
\begin{rmenum}
\item \label{item:ge:barrier} $A(G)$ is a barrier for which $D_{A(G)} = D(G)$ and $C_{A(G)} = C(G)$; 
\item \label{item:ge:fc} each odd component of $G-A(G)$ is factor-critical; and, 
\item \label{item:ge:allowed}  every edge in $E_G[A(G), D(G)]$ is allowed, 
whereas no edge in $E_G[A(G), A(G)\cup C(G)]$ is allowed. 
\end{rmenum} 
\end{theorem}

\subsection{Factor-Components}

An edge is {\em allowed} if it is contained in some maximum matching. 
Two vertices are {\em factor-connected} 
if they are connected by a path whose edges are allowed. 
A subgraph is {\em factor-connected} 
if any two vertices are factor-connected. 
A maximal factor-connected subgraph 
is called a {\em factor-connected component} or {\em factor-component}. 
A graph consists of its factor-components and edges joining them that are not allowed. 
The set of factor-components of $G$ is denoted by $\ecomp{G}$.

A factor-component $C$ is {\em inconsistent} if $V(C)\cap D(G) \neq \emptyset$. 
Otherwise, $C$ is said to be {\em consistent}. 
We denote the sets of consistent and inconsistent factor-components of $G$ 
by $\const{G}$ and $\inconst{G}$, respectively. 
The next property is easily confirmed from the Gallai-Edmonds structure theorem.

\begin{fact} \label{fact:inconst2ge}
A subgraph $C$ of $G$ is an inconsistent factor-component 
if and only if 
$C$ is a connected component of $G[D(G)\cup A(G)]$. 
Any consistent factor-component has the vertex set contained in $C(G)$. 
\end{fact} 

That is, the structure of inconsistent factor-components 
are rather trivial under the Gallai-Edmonds structure theorem.

\section{Basilica Decomposition of Graphs} \label{sec:basilica}

\subsection{Central Concepts} \label{sec:basilica:main}

In Section~\ref{sec:basilica}, we introduce the basilica decomposition of graphs~\cite{DBLP:conf/isaac/Kita12, DBLP:journals/corr/abs-1205-3816}. 
The theory of basilica decomposition is made up of 
the three central concepts: 
\begin{rmenum}
\item a canonical partial order between factor-components (Theorem~\ref{thm:order}), 
\item the  general Kotzig-Lov\'asz decomposition (Theorem~\ref{thm:gkl}), and 
\item an interrelationship between the two (Theorem~\ref{thm:cor}). 
\end{rmenum} 
In Section~\ref{sec:basilica:main}, 
we explain these three concepts 
and give the definition of the basilica decomposition. 
Every statement in the following  are from Kita~\cite{DBLP:conf/isaac/Kita12, DBLP:journals/corr/abs-1205-3816}%
~\footnote{The essential part of the structure described by the basilica decomposition 
lies in the factorizable graph $G[C(G)]$. 
Therefore, statements for factorizable graphs~\cite{DBLP:conf/isaac/Kita12, DBLP:journals/corr/abs-1205-3816}  
can be straightforwardly generalized for arbitrary graphs 
under the Gallai-Edmonds structure theorem.} 
In Section~\ref{sec:basilica}, let $G$ be a graph unless otherwise stated.

\begin{definition} 
We call a set $X\subseteq V(G)$ {\em separating} 
if it is the disjoint union of  vertex sets of some factor-components.
For $G_1, G_2\in\ecomp{G}$, 
we say $G_1 \yield G_2$ 
if there exists a separating set $X\subseteq V(G)$ with $V(G_1)\cup V(G_2)\subseteq X$ 
such that $G[X]/G_1$ is a factor-critical graph. 
\end{definition} 

\begin{theorem} \label{thm:order} 
For a graph $G$, 
the binary relation $\yield$ is a partial order over $\ecomp{G}$. 
\end{theorem} 

\begin{definition} 
For $u, v\in V(G)\setminus D(G)$, 
we say $u\gsim{G} v$ if $u$ and $v$ are identical or 
if $u$ and $v$ are factor-connected and satisfy $\defi{G-u-v} > \defi{G}$. 
\end{definition} 

\begin{theorem} \label{thm:gkl} 
For a graph $G$, the binary relation $\gsim{G}$ is an equivalence relation. 
\end{theorem} 

We denote as $\gpart{G}$ the family of equivalence classes determined by $\gsim{G}$. 
This family is known as the {\em general Kotzig-Lov\'asz decomposition} 
or just the {\em Kotzig-Lov\'asz decomposition} of $G$. 
From the definition of $\gsim{G}$, 
for each $H\in\ecomp{G}$,  the family $\{ S\in\gpart{G}: S\subseteq V(H) \}$ forms a partition of $V(H)\setminus D(G)$. 
We denote this family by $\pargpart{G}{H}$. 

Let $H\in \ecomp{G}$. 
The sets of strict and nonstrict upper bounds of $H$ are denoted 
by $\parup{G}{H}$ and $\parupstar{G}{H}$, respectively. 
The sets of vertices $\bigcup \{ V(I) : I\in\parup{G}{H}\}$  
and $\bigcup \{ V(I): I \in\parupstar{G}{H}\}$ are denoted by $\vparup{G}{H}$ 
and $\vparupstar{G}{H}$, respectively. 

\begin{theorem} \label{thm:cor}
Let $G$ be a graph, and let $H\in\ecomp{G}$. 
Then, for each connected component $K$ of $G[\vparup{G}{H}]$, 
there exists $S\in \pargpart{G}{H}$ such that 
$\parNei{G}{K} \cap V(H) \subseteq S$. 
\end{theorem} 

Under Theorem~\ref{thm:cor}, 
for $S\in\pargpart{G}{H}$, 
we denote by $\parup{G}{S}$ 
the set of factor-components that are contained in a connected component $K$ of $G[\vparup{G}{H}]$ 
 with $\parNei{G}{K} \cap V(H) \subseteq S$. 
The set $\bigcup \{ V(I): I \in\vparup{G}{H}\}$ is denoted by $\vparup{G}{S}$. 
We denote $\vparup{G}{H} \setminus S \setminus \vparup{G}{S}$ by $\parvcoup{G}{S}$. 

Theorem~\ref{thm:cor} integrates the two structures given by Theorems~\ref{thm:order} 
and \ref{thm:gkl} into a structure of graphs 
that is reminiscent of an architectural building. 
We call this integrated structure the {\em basilica decomposition} of a graph.

\subsection{Remark on Inconsistent Factor-Components} 

Inconsistent factor-components in a graph have a trivial structure regarding the basilica decomposition. 
The next statement is  easily confirmed from  Fact~\ref{fact:inconst2ge} and the Gallai-Edmonds structure theorem.

\begin{fact}\label{fact:inconst2basilica}
Let $G$ be a graph. 
Any inconsistent component is minimal in the poset $(\ecomp{G}, \yield)$. 
For any $H\in\inconst{G}$, 
if $V(H)\cap A(G)\neq \emptyset$, then $\pargpart{G}{H} = \{ V(H)\cap A(G)\}$; 
otherwise, $\pargpart{G}{H} = \emptyset$. 
\end{fact} 

For simplicity, 
even for $H\in\inconst{G}$ with  $V(H)\cap A(G)  = \emptyset$, 
we treat as if $V(H)\cap A(G)$ is a member of $\gpart{G}$. 
That is, we  let $\pargpart{G}{H} = \{V(H)\cap A(G)\}$ 
and 
$\parvcoup{G}{V(H)\cap A(G)} = \parvcoup{G}{\emptyset} = V(H)\cap D(G) = V(H)$.

Under Fact~\ref{fact:inconst2basilica}, the substantial information provided by the basilica decomposition lies in the consistent factor-components.

\subsection{Additional Properties} 

In this section, we present some properties of the basilica decomposition 
that are used in later sections. 

\begin{lemma}[Kita~\cite{DBLP:conf/cocoa/Kita13}] \label{lem:basilica2path} 
Let $G$ be a graph, and let $M$ be a maximum matching of $G$. 
Let $H\in\const{G}$,   $S\in\pargpart{G}{H}$, and $s\in S$. 
\begin{rmenum} 
\item 
 For any $t\in S$, 
 there is an $M$-forwarding path from $s$ to $t$, 
 whose vertices are contained in $S\cup \parvcoup{G}{S}$; 
 however, there is no $M$-saturated path between $s$ and $t$. 
\item 
For any $t\in \parvcoup{G}{S}$, 
there exists an $M$-saturated path between $s$ and $t$ 
whose vertices are contained in $S\cup \parvcoup{G}{S}$.  
\item 
For any $t\in \vparup{G}{S}$, 
there is an $M$-forwarding path from $t$ to $s$,  
whereas there is no $M$-forwarding path from $s$ to $t$ 
or $M$-saturated path between $s$ and $t$. 
\end{rmenum} 
\end{lemma}  

The first part of the next lemma is provided in Kita~\cite{kita2015graph}, 
and the second part can be easily proved from Lemma~\ref{lem:basilica2path}.

\begin{lemma} \label{lem:noear} 
Let $G$ be a graph, and let $M$ be a maximum matching of $G$. 
Let $S\in\gpart{G}$. 
If there is an $M$-ear relative to $S\cup \parvcoup{G}{S}$ that has internal vertices, 
then the ends of this ear are contained in $S$. 
\end{lemma}

\section{Poset with Transitive Forbidden Relation} \label{sec:ptfp}

We now introduce the new concept of {\em posets with a transitive forbidden relation}, 
which serves as a language to describe 
the nonbipartite DM decomposition. 

\begin{definition} 
Let $X$ be a set, and let $\preceq$ be a partial order over $X$. 
Let $\smile$ be a binary relation over $X$ 
such that, 
\begin{rmenum} 
\item 
for each $x,y,z \in X$, if    $x \preceq y$ and $y\smile z$ hold, then  $x \smile z$ holds (transitivity); 
\item for each $x \in X$, $x\smile x$ does not hold (nonreflexivity); and, 
\item for each $x, y\in X$, if $x\smile y$ holds, then $y\smile x$ also holds (symmetry). 
\end{rmenum} 
We call this poset endowed with this additional binary relation 
a {\em poset with a transitive forbidden relation} or {\em TFR poset}  in short, 
and denote this by $(X, \preceq, \smile)$. 
We call a pair of two elements $x$ and $y$ with $x \smile y$ {\em forbidden}. 
\end{definition}

Let $(X, \preceq, \smile)$ be a TFR poset. 
For two elements $x, y\in X$ with $x \smile y$, 
we say that $x \overset{\mathrm{\star}}{\smile} y$ 
if, there is no $z\in X\setminus \{x, y\}$ with $x\preceq z$ and $z\smile y$. 
We call such a forbidden pair of $x$ and $y$ {\em immediate}. 
A TFR poset can be visualized in a similar way to an ordinary posets. 
We represent $\preceq$ just in the same way as the Hasse diagrams 
and depict $\smile$ by indicating every immediate forbidden pairs.

\begin{definition} 
Let $P$ be a TFR poset $(X, \preceq, \smile)$. 
A lower or upper ideal $Y$ of $P$ is {\em legitimate} 
if no elements $x, y\in Y$ satisfy $x\smile y$. 
Otherwise, we say that $Y$ is {\em illegitimate}. 
Let $Y$ be a consistent lower or upper ideal, 
and let $Z$ be the subset of $X\setminus Y$ such that, 
for each $x\in Z$, there exists $y\in Y$ with $x\smile y$. 
We say that $Y$ is {\em spanning} if $Y \cup Z = X$. 
\end{definition}

\section{Dulmage-Mendelsohn Decomposition for General Graphs} \label{sec:gdm}

We now provide our new results of the DM decomposition for general graphs. 
In this section, unless otherwise stated, let $G$ be a graph. 

\begin{definition} 
A {\em  Dulmage-Mendelsohn component}, or a {\em  DM component } in short, 
 is a subgraph of the form $G[S\cup \parvcoup{G}{S}]$, where $S \in \gpart{G}$, 
endowed with $S$ as an attribute known as the {\em base}.  
For a  DM component $C$, the base of $C$ is denoted by $\sd{C}$. 
Conversely, for $S\in \gpart{G}$, $\ds{S}$ denotes 
the  DM components whose base is $S$.  
We denote by $\dmcomp{G}$ the set of  DM components of $G$. 
\end{definition}

Hence, distinct  DM components can be equivalent as a subgraph of $G$. 
A base $S\in \gpart{G}$ uniquely determines a  DM component.

\begin{definition} 
A  DM component $C$ is said to be {\em inconsistent} 
if $\sd{C} \in\pargpart{G}{H}$ for some $H\in\inconst{G}$; 
otherwise, $C$ is said to be {\em consistent}. 
The sets of consistent and inconsistent   DM components 
are denoted by $\cdmcomp{G}$ and $\idmcomp{G}$, respectively. 
\end{definition}

Under Fact~\ref{fact:inconst2basilica}, 
any $H\in\idmcomp{G}$ is equal to an inconsistent factor-component 
as a subgraph of $G$,  
and $\sd{H} = V(H)\cap A(G)$ and $V(H)\setminus \sd{H} = V(H)\cap D(G)$.

\begin{definition} 
We define binary relations $\gdmo$ and $\gdm$ over $\dmcomp{G}$ as follows: 
for $D_1, D_2\in\dmcomp{G}$, 
we let $D_1 \gdmo D_2$ if $D_1 = D_2$ or if 
$\parNei{G}{\parvcoup{G}{S_1}}\cap S_2 \neq \emptyset$; 
we let $D_1 \gdm D_2$ if 
there exist $C_1, \ldots, C_k\in\dmcomp{G}$, where $k\ge 1$, 
such that $\sd{C_1} = \sd{D_1}$, $\sd{C_k} = \sd{D_2}$, and $C_i \gdmo C_{i+1}$ for each $i \in \{1,\ldots, k\}\setminus \{k\}$. 
\end{definition}

\begin{definition} 
We define binary relations $\nono$ and $\non$ over $\dmcomp{G}$ as follows: 
for $D_1, D_2\in\dmcomp{G}$, 
we let $D_1 \nono D_2$ if $\sd{D_2}\subseteq V(D_1)\setminus \sd{D_1}$ holds; 
we let $D_1 \non D_2$ 
if there exists $D' \in \dmcomp{G}$ 
with $D_1 \gdm D'$ and $D' \nono D_2$. 
\end{definition}

In the following, we prove that $(\dmcomp{G}, \gdm, \non)$ is a TFR poset, 
which gives a generalization of the DM decomposition. 
The next lemma is easily observed from Facts~\ref{fact:inconst2ge} and \ref{fact:inconst2basilica}. 

\begin{lemma} \label{lem:inconst2dm} 
If $C$ is an inconsistent  DM component of a graph $G$, 
then there is no $C'\in\dmcomp{G}\setminus \{C\}$ 
with $C \gdm C'$ or $C \non C'$. 
\end{lemma}

We first prove that $\gdm$ is a partial order over $\dmcomp{G}$. 
We provide Lemmas~\ref{lem:disjoint2path} and \ref{lem:acyclic} 
and thus prove Theorem~\ref{thm:gdmorder}.

\begin{lemma} \label{lem:disjoint2path} 
Let $G$ be a graph, let $M$ be a maximum matching of $G$, 
and let $D_1, \ldots, D_k\in\dmcomp{G}$, where $k\ge 1$, 
be DM components with $D_1\gdmo \cdots \gdmo D_k$ 
no two of which share vertices.  
Then, for any $s\in \sd{D_1}$ and for any $t\in \sd{D_k}$, 
there is an $M$-forwarding path from $s$ to $t$ 
whose vertices are contained in $V(D_1)\dot\cup \cdots \dot\cup V(D_k)$. 
If $D_k\in\cdmcomp{G}$ holds and $t$ is a vertex from $V(D_k)\setminus \sd{D_k}$,  
then there is an $M$-saturated path between $s$ and $t$ 
whose vertices are contained in $V(D_1)\dot\cup \cdots \dot\cup V(D_k)$.  
\end{lemma} 

\begin{proof} 
For each $i \in \{1,\ldots, k\}\setminus \{k\}$, 
let $t_i\in \parvcoup{G}{\sd{D_i}}$ and $s_{i+1}\in \sd{D_{i+1}}$ 
be vertices with $t_is_{i+1}\in E(G)$.  
Let $s_1 := s$ and $t_k := t$. 
According to Lemma~\ref{lem:basilica2path}, 
for each $i\in \{1,\ldots, k\}\setminus \{k\}$, 
there is an $M$-saturated path $P_i$ between $s_i$ and $t_i$ 
with $V(P_i)\subseteq V(D_i)$; 
additionally, 
there is an $M$-forwarding path $P_k$ from $s_k$ to $t$ with $V(P_k)\subseteq V(D_k)$. 
Thus, $P_1 + \cdots + P_k + \{ t_is_{i+1}: i = 1,\ldots, k-1\}$
 is a desired $M$-forwarding path from $s$ to $t$. 
The claim for $t\in V(D_k)\setminus \sd{D_k}$ 
can be also proved in a similar way using Lemma~\ref{lem:basilica2path}. 
\end{proof}

Lemma~\ref{lem:disjoint2path} yields Lemma~\ref{lem:acyclic}: 

\begin{lemma} \label{lem:acyclic} 
Let $G$ be a graph, let $M$ be a maximum matching of $G$, 
and let $D_1, \ldots, D_k$, where $k\ge 2$, 
be   DM components with $D_1\gdmo \cdots \gdmo D_k$ 
such that $\sd{D_i} \neq \sd{D_{i+1}}$ for any $i\in\{1,\ldots, k-1\}$. 
Then, for any $i,j\in \{1,\ldots, k\}$ with $i\neq j$, 
$V(D_i) \cap V(D_j) = \emptyset$.  
\end{lemma} 

\begin{proof} 
Let $q$ be the minimum number from $\{1,\ldots, k-1\}$ 
such that $D_{q+1}$ shares vertice with some $D_i$, where $i\in \{1,\ldots, q-1\}$. 
Let $p$ be the maximum number from $\{1, \ldots, q-1\}$ 
such that $V(D_{q+1})\cap V(D_p)\neq \emptyset$. 
Then, $D_p, \ldots, D_q$ are mutually disjoint. 
Additionally, from Lemma~\ref{lem:inconst2dm}, 
we have $D_p, \ldots, D_q \in \cdmcomp{G}$. 
Either $\sd{D_{q+1}} \subseteq V(D_p)$ or $\parvcoup{G}{\sd{D_{q+1}}} \cap V(D_p) \neq \emptyset$ holds. 
In the first case, 
let $t_{q+1}$ be an arbitrary vertex from $\parvcoup{G}{\sd{D_{q+1}}}$, 
and, in the second case, 
let $t_{q+1}$ be a vertex from $\parvcoup{G}{\sd{D_{q+1}}} \cap V(D_p)$. 
Let $t_q \in \parvcoup{G}{\sd{D_q}}$ and $s_{q+1} \in \sd{D_{q+1}}$ 
be vertices with $t_qs_{q+1} \in E(G)$. 
From Lemma~\ref{lem:basilica2path}, 
there is an $M$-saturated path $P$ between $s_{q+1}$ and $t_{q+1}$ with $V(P)\subseteq V(D_{q+1})$. 

Let $t_{p} \in \parvcoup{G}{\sd{D_p}}$ and $s_{p+1}$ 
be vertices with $t_ps_{p+1} \in E(G)$.  
From Lemma~\ref{lem:disjoint2path}, 
there is an $M$-saturated path $Q$ between $s_{p+1}$ and $t_q$ 
with $V(Q) \subseteq V(D_p)\cup \cdots \cup V(D_q)$. 

Trace $P$ from $s_{q+1}$, 
and let $x$ be the first encountered vertex in $D_p$, 
for which $x = s_{q+1}$ is allowed. 
Then, by letting $R:= t_ps_{p+1} + Q + t_qs_{q+1} + s_{q+1}Px$, $R$ is an $M$-ear relative to $D_p$ 
whose ends are $x$ and $t_p$. 
Note that $R$ contains internal vertices, e.g., $s_{p+1}$.

If $x\in \sd{D_p}$ holds, then, under Lemma~\ref{lem:basilica2path}, 
let $L$ be an $M$-saturated path between $x$ and $t_p$ 
with $V(L)\subseteq V(D_p)$. 
Then, $R + L$ is an $M$-alternating circuit containing non-allowed edges of $G$. 
This is a contradiction. 
If $x\in \parvcoup{G}{\sd{D_p}}$ holds, 
then this contradicts Lemma~\ref{lem:noear}. 
This completes the proof. 
\end{proof}

Combining Lemmas~\ref{lem:disjoint2path} and \ref{lem:acyclic}, 
the next lemma can be stated, which we do not use for proving Theorem~\ref{thm:gdmorder}.  

\begin{lemma} \label{lem:order2path} 
Let $G$ be a graph. 
Let $C_1, C_2 \in \dmcomp{G}$ with $C_1 \gdm C_2$, and let $D_1, \ldots, D_k\in\dmcomp{G}$, where $k\ge 1$, 
be   DM components with $C_1 = D_1$, $C_2 = D_k$, and $D_1\gdmo \cdots \gdmo D_k$.  
Then, for any $s\in \sd{D_1}$ and for any $t\in \sd{D_k}$ (resp. $t\in V(D_k)\setminus \sd{D_k}$), 
there is an $M$-forwarding path from $s$ to $t$ 
(resp. $M$-saturated path between $s$ and $t$) 
whose vertices are contained in $V(D_1)\dot\cup \cdots \dot\cup V(D_k)$.  
\end{lemma}

We now obtain Theorem~\ref{thm:gdmorder}.

\begin{theorem} \label{thm:gdmorder} 
Let $G$ be a graph. 
Then, $\gdm$ is a partial order over $\dmcomp{G}$. 
\end{theorem} 

\begin{proof} 
Reflexivity and transitivity are obvious from the definition. 
Antisymmetry is obviously implied by Lemma~\ref{lem:acyclic}. 
\end{proof}

In the following, we prove the properties required for $\non$ 
to form a TFR poset $(\dmcomp{G}, \gdm, \non)$. 
We provide Lemmas~\ref{lem:straight}, \ref{lem:path2order}, and \ref{lem:non2sym}, 
and thus prove Theorem~\ref{thm:gdm}

\begin{lemma} \label{lem:straight} 
Let $G$ be a graph, and let $M$ be a maximum matching of $G$. 
Let $s, t \in V(G)$, and let $S$ be the member of $\gpart{G}$ with $s\in S$. 
Let $P$ be an $M$-forwarding path $P$ from $s$ to $t$ 
or an $M$-saturated path between $s$ and $t$. 
If $t\in S\cup \parvcoup{G}{S}$ holds, 
then $P - E(G[S\cup \parvcoup{G}{S}])$ is empty; 
otherwise, $P - E(G[S\cup \parvcoup{G}{S}])$ is a path. 
\end{lemma} 

\begin{proof} 
Suppose that the claim fails. 
The connected components of $P - E(G[S\cup \parvcoup{G}{S}])$ 
except for the one that contains $s$ 
are $M$-ears relative to $S\cup \parvcoup{G}{S}$ with internal vertices. 
Let $S'$ be the set of the ends of these $M$-ears. 
From Lemma~\ref{lem:noear}, we have $S' \subseteq S$. 
Trace $P$ from $s$, and let $s'$ be the first vertex in $S'$. 
Then, $sPr$ is an $M$-saturated path between $s$ and $s'$, 
which contradicts $s \gsim{G} s'$. 
This proves the claim. 
\end{proof}

Lemma~\ref{lem:straight} derives the next lemma 
with Lemmas~\ref{lem:basilica2path} and \ref{lem:noear}. 

\begin{lemma} \label{lem:path2order}
Let $G$ be a graph, and let $M$ be a maximum matching of $G$.  
Let $s, t\in V(G)$, 
and let $S$ and  $T$ be the members from $\gpart{G}$  with $s\in S$ and $t\in T$, respectively. 
\begin{rmenum} 
\item \label{item:nosat} 
If there is no $M$-saturated path between $s$ and $t$, 
whereas there is an $M$-forwarding path from $s$ to $t$, 
then $\ds{S} \gdm \ds{T}$ holds. 
\item \label{item:sat} 
If there is an $M$-saturated path between $s$ and $t$, 
then $\ds{S} \non \ds{T}$ holds. 
\end{rmenum} 
\end{lemma} 

\begin{proof} 
Let $P$ be an $M$-forwarding path from $s$ to $t$ 
or an $M$-saturated between $s$ and $t$. 
We proceed by induction on the number of edges in $P$. 
If $V(P)\subseteq S\cup \parvcoup{G}{S}$ holds, then Lemma~\ref{lem:basilica2path} proves the statements. 
Hence, let $V(P)\setminus V(\ds{S}) \neq \emptyset$, 
and assume that the statements hold for every case where $|E(P)|$ is fewer. 
By Lemma~\ref{lem:straight}, 
$P-E(\ds{S})$ is an $M$-exposed path one of whose end is $t$; 
let $x$ be the other end of $P$.  
Let $y \in V(P)$ be the vertex with $xy \in E(P)$. 
The subpath $sPx$ is obviously $M$-saturated between $s$ and $x$, 
for which $x\in V(\ds{S})$ holds. 
Hence, Lemma~\ref{lem:basilica2path} implies $x\in \parvcoup{G}{S}$. 
Let $R$ be the member of $\gpart{G}$ with $y\in R$. 
Then, we have $\ds{S} \gdmo \ds{R}$.  

If $P$ is an $M$-saturated path, 
then $yPt$ is an $M$-saturated path between $y$ and $t$. 
Therefore, the induction hypothesis implies $\ds{R} \non \ds{T}$. 
Thus, $\ds{S} \non \ds{T}$ is obtained, 
and \ref{item:sat} is proved. 

Consider now the case where $P$ is an $M$-forwarding path from $s$ to $t$. 
The subpath $yPt$ is an $M$-forwarding path from $y$ to $t$. 
\begin{pclaim} 
There is no $M$-saturated path between $y$ and $t$ in $G$. 
\end{pclaim} 
\begin{proof} 
Suppose the contrary, 
and let $Q$ be an $M$-saturated path between $y$ and $t$.  
First, suppose that $Q$ shares vertices with $S\cup \parvcoup{G}{S}$. 
Trace $Q$ from $y$, and let $z$ be the first vertex in $S\cup \parvcoup{G}{S}$. 
Then, $zQy + yx$ is an $M$-ear relative to $S\cup \parvcoup{G}{S}$ 
with internal vertices, e.g., $y$.  
This contradicts Lemma~\ref{lem:noear}. 
Hence, $Q$ is disjoint from  $S\cup \parvcoup{G}{S}$.  
This however implies that $sPx + xy + Q$ is an $M$-saturated path between $s$ and $t$, 
which contradicts the assumption. 
\end{proof}  

Therefore, under the induction hypothesis, 
$\ds{R} \gdm \ds{T}$. 
We thus obtain $\ds{S} \gdm \ds{T}$, 
and \ref{item:nosat} is proved. 
\end{proof}

The symmetry of $\non$ now can be proved from Lemmas~\ref{lem:order2path} and \ref{lem:path2order}. 

\begin{lemma} \label{lem:non2sym}
For a graph $G$, 
the binary relation $\non$ is symmetric, that is, 
if $D_1 \non D_2$ holds for $D_1, D_2 \in \dmcomp{G}$, 
then $D_2 \non D_1$ holds. 
\end{lemma} 

\begin{proof} 
Let $M$ be a maximum matching of $G$, 
and let $x_1\in \sd{D_1}$ and $x_2\in\sd{D_2}$. 
From Lemma~\ref{lem:order2path}, 
$D_1\non D_2$ implies that 
there is an $M$-saturated path $P$ between $x_1$ and $x_2$. 
From Lemma~\ref{lem:path2order}, 
this implies $D_2 \non D_1$. 
\end{proof}

We now prove Theorem~\ref{thm:gdm} from Theorem~\ref{thm:gdmorder} and Lemma~\ref{lem:non2sym}: 

\begin{theorem} \label{thm:gdm} 
For a graph $G$, 
the triple $(\dmcomp{G}, \gdm, \smile)$ is a TFR poset. 
\end{theorem}

\begin{proof}
Under Theorem~\ref{thm:gdmorder}, 
it  now suffices prove the conditions for $\non$. 
Nonreflexivity and transitivity are obvious from the definition. 
Symmetry is proved by Lemma~\ref{lem:non2sym}. 
\end{proof}

For a graph $G$, the TFR poset $(\dmcomp{G}, \gdm, \smile)$ is uniquely determined. 
We denote this TFR poset by $\gdmposet{G}$. 
We call this canonical structure of a graph $G$ 
that the TIP $\gdmposet{G}$ describes the {\em nonbipartite Dulmage-Mendelsohn \textup{(}DM\textup{)} decomposition} 
of $G$. 
We show in Section~\ref{sec:original} that this is a generalization 
of the classical DM decomposition for bipartite graphs. 

\begin{remark} 
As mentioned previously, 
a DM component is identified by its base. 
Therefore, the nonbipartite DM decomposition 
is essentially the relations between the members of $\gpart{G}$. 
\end{remark} 

\section{Preliminaries on Maximal Barriers} \label{sec:barrierpre}

\subsection{Classical Properties of Maximal Barriers} 
We now present some preliminary properties of maximal barriers 
to be used in Section~\ref{sec:barrier}. 
A barrier is {\em maximal} if it is  inclusion-wise maximal. 
A barrier $X$ is {\em odd-maximal} 
if it is maximal with respect to $D_X$; 
that is, for no $Y\subseteq D_X$, $X \cup Y$ is a barrier. 
A maximal barrier is an odd-maximal barrier.

The next two propositions are classical facts. 
See Lov\'asz and Plummer~\cite{lovasz2009matching}. 

\begin{proposition}\label{prop:oddmaximal2fc}
Let $G$ be a graph, and let $X\subseteq V(G)$ be a barrier. 
Then, $X$ is an odd-maximal barrier 
if and only if every odd component of $G-X$ are factor-critical. 
\end{proposition} 

\begin{proposition}\label{prop:oddmaximal2maximal}
Let $G$ be a graph. 
An odd-maximal barrier is a maximal barrier 
if and only if $C_X = \emptyset$. 
\end{proposition}

\subsection{Generalization of Lov\'asz's Canonical Partition Theorem} 

In this section, we explain a known theorem about 
the structure of a given odd-maximal barrier~\cite{DBLP:conf/cocoa/Kita13}. 
This theorem is a generalization of {\em Lov\'asz's canonical partition theorem}~\cite{lovasz1972structure, lovasz2009matching, DBLP:conf/cocoa/Kita13} for general graphs, which is originally for factor-connected graphs. 
This  theorem contains the classical result about the relationship between 
maximal barriers and the Gallai-Edmonds decomposition, 
which states that $A(G)$ of a graph $G$ is the intersection of all maximal barriers~\cite{lovasz2009matching}. 

\begin{theorem}[Kita~\cite{DBLP:conf/cocoa/Kita13}] \label{thm:gl} 
Let $G$ be a graph and $X\subseteq V(G)$ be an odd-maximal barrier of $G$. 
Then, there exist $S_1,\ldots, S_k\in\gpart{G}$, where $k\ge 1$, 
such that 
$X = S_1 \dot\cup \cdots \dot\cup S_k$ 
and $D_X = \parvcoup{G}{S_1} \dot\cup \cdots \dot\cup \parvcoup{G}{S_k}$. 
The odd components of $G-X$ are 
the connected components of $G[\parvcoup{G}{S_i}]$, where $i$ is taken over all $\{1,\ldots, k\}$. 
\end{theorem} 

The next statement can be derived from Theorem~\ref{thm:gl} as a corollary. 

\begin{corollary} \label{cor:gkl}
Let $G$ be a graph. 
For each $S\in\gpart{G}$, 
$G[\parvcoup{G}{S}]$ consists of $|S| + \defi{G[S \cup \parvcoup{G}{S}]}$ connected components, 
which are  factor-critical. 
If $S\in\pargpart{G}{H}$ holds for some $H\in\const{G}$, 
then $\defi{G[S \cup \parvcoup{G}{S}]} = 0$; 
otherwise, $\defi{G[S \cup \parvcoup{G}{S}]} > 0$. 
Let $\mathcal{S} := \bigcup \{  S \in \pargpart{G}{H} : H\in\inconst{G} \mbox{ and } V(H)\cap X \neq \emptyset \}$. 
Then, $\Sigma_{S \in \mathcal{S}}  \defi{G[S \cup \parvcoup{G}{S}]}  = \defi{G}$. 
\end{corollary}

\section{Canonical Characterization of Maximal Barriers} \label{sec:barrier} 

We now derive the characterization of the family of maximal barriers 
in general graphs, using the nonbipartite DM decomposition. 
In this section, unless otherwise stated, let $G$ be a graph. 
It is a known fact that a graph has an exponentially many number of maximal barriers, 
however the family of maximal barriers can be fully characterized 
in terms of ideals of $\gdmposet{G}$.

\begin{definition} 
For $\mathcal{I}\subseteq \dmcomp{G}$, 
 the {\em normalization} of $\mathcal{I}$ 
is the set $\mathcal{I} \cup \idmcomp{G}$. 
A set $\mathcal{I}' \subseteq \dmcomp{G}$ is said to be {\em normalized} 
if $\mathcal{I'} = \mathcal{I} \cup \idmcomp{G}$ 
for some $\mathcal{I} \subseteq \dmcomp{G}$. 
\end{definition}

From Lemma~\ref{lem:inconst2dm}, 
the next statement can be easily observed. 

\begin{observation} \label{obs:norm} 
The normalization of an upper ideal is an upper ideal. 
The normalization of a legitimate upper ideal is legitimate. 
\end{observation} 

From Theorem~\ref{thm:gl}, 
the next lemma characterizes the family of odd-maximal barriers. 

\begin{lemma} \label{lem:oddmaximal} 
Let $G$ be a graph. 
A set of vertices $X\subseteq V(G)$ is an odd-maximal barrier 
if and only if 
there exists a  legitimate normalized upper ideal $\mathcal{I}$ 
of the TFR poset $\gdmposet{G}$ 
such that $X = \bigcup \{ \sd{C} : C\in\mathcal{I} \}$. 
\end{lemma} 
\begin{proof} 
We first prove the sufficiency. 
Let $X$ be an odd-maximal barrier, and, under Theorem~\ref{thm:gl}, let $S_1,\ldots, S_k$, where $k\ge 1$, 
be the members of $\gpart{G}$ such that $X = S_1\cup \cdots \cup S_k$.
Let $\mathcal{I} := \{\ds{S_i} : i = 1,\ldots, k\}$. 
We prove that $\mathcal{I}$ is a legitimate normalized upper ideal of $\gdmposet{G}$. 
For proving $\mathcal{I}$ is an upper ideal,  
it suffices to prove that,  for any $C\in\dmcomp{G}$,  
$\ds{S_i} \gdmo C$ implies $\sd{C} \subseteq X$;  
and, this is obviously confirmed from Theorem~\ref{thm:gl}.  
It is also confirmed by Theorem~\ref{thm:gl} that this upper ideal is normalized and legitimate.  

Next, we prove the necessity. 
Let $\mathcal{I}$ be a legitimate normalized upper ideal of $\gdmposet{G}$, 
and let $X = \bigcup \{ \sd{C} : C\in\mathcal{I} \}$. 
From the definition of $\gdmo$, 
$\mathcal{I}$ being an upper ideal implies 
that, for each $C\in\mathcal{I}$, 
$\parNei{G}{\parvcoup{G}{\sd{C}}} \subseteq X$; 
$\mathcal{I}$ being legitimate implies 
that $\parvcoup{G}{\sd{C}} \cap X = \emptyset$. 
Hence, each connected component of $G[\parvcoup{G}{\sd{C}}]$ 
is also a connected component of $G-X$ that is factor-critical.  
Therefore, 
Corollary~\ref{cor:gkl} 
implies that $G-X$ has $|X| + \defi{G}$ odd components, 
and accordingly $X$ is a barrier. 
By Theorem~\ref{thm:gl}, these odd components are factor-critical, 
and therefore, Proposition~\ref{prop:oddmaximal2fc} implies that $X$ is odd-maximal. 
\end{proof}

From Lemma~\ref{lem:oddmaximal} and Proposition~\ref{prop:oddmaximal2maximal}, 
the family of maximal barriers is now characterized:

\begin{theorem} \label{thm:maximalbarrier}
Let $G$ be a graph. 
A set of vertices $X\subseteq V(G)$ is a maximal barrier 
if and only if 
there exists a  spanning legitimate normalized upper ideal $\mathcal{I}$ 
of the TFR poset $\gdmposet{G}$ 
such that $X = \bigcup \{ \sd{C} : C\in\mathcal{I} \}$. 
\end{theorem}

\section{Original DM Decomposition for Bipartite Graphs} \label{sec:original}

In this section, we explain the original DM decomposition for bipartite graphs, 
and prove this from our result in Section~\ref{sec:gdm}.  
In the remainder of this section, unless stated otherwise, 
let $G$ be a bipartite graph with color classes $A$ and $B$, 
and let $W\in \{A, B\}$. 

\begin{definition} 
The binary relations $\dmo{W}$ and $\dm{W}$ over $\ecomp{G}$ are defined as follows: 
for $G_1, G_2\in\ecomp{G}$, 
let $G_1 \dmo{W} G_2$ if $G_1 = G_2$ or if $E_G[W\cap V(G_2), V(G_1)\setminus W]\neq \emptyset$; 
let $G_1 \dm{W} G_2$ if 
there exist $H_1, \ldots, H_k\in\ecomp{G}$, where $k\ge 1$, 
such that $H_1 = G_1$, $H_k = G_2$, and $H_1 \dmo{W} \cdots \dmo{W} H_k$. 
\end{definition}  

Note that $G_1 \dm{A} G_2$ holds if and only if $G_2 \dm{B} G_1$ holds. 
The next theorem determines the classical DM decomposition. 

\begin{theorem}[Dulmage and Mendelsohn~\cite{dm1958, dm1959, dm1963, lovasz2009matching}]\label{thm:dm}
Let $G$ be a bipartite graph with color classes $A$ and $B$, 
and let $W\in \{A, B\}$. 
Then, the binary relation $\dm{W}$ is a partial order over $\ecomp{G}$. 
\end{theorem}

We call the poset $(\ecomp{G}, \dm{W})$ proved by Theorem~\ref{thm:dm} 
the {\em Dulmage-Mendelsohn decomposition} of 
a bipartite graph $G$.

In the following, 
we demonstrate how our nonbipartite DM decomposition derives Theorem~\ref{thm:dm} 
under the special properties of  bipartite graphs  
regarding  
\begin{rmenum}
\item 
inconsistent factor-components (Observation~\ref{obs:dm2inconst}) 
and 
\item 
the basilica decomposition (Observation~\ref{obs:bi2basilica}). 
\end{rmenum}
The set of inconsistent factor-components with some vertices from $D(G) \setminus W$ 
is denoted by $\ainconst{G}{W}$. 
The next statement about $\ainconst{G}{W}$ 
can be easily confirmed from the Gallai-Edmonds structure theorem. 
This statement can also be proved from first principles 
by a simple discussion on alternating paths, which is employed in original proof. 
As is also the case in the basilica and nonbipartite DM decomposition, 
the substantial part of the bipartite DM decomposition 
lies in $\const{G}$.

\begin{observation} \label{obs:dm2inconst}
The sets  $\ainconst{G}{A}$ and $\ainconst{G}{B}$ are disjoint. 
Any $C \in \ainconst{G}{B}$ is minimal with respect to $\dm{A}$. 
\end{observation}

Bipartite graphs have a trivial structure 
regarding the basilica decomposition: 

\begin{observation} \label{obs:bi2basilica}
Let $G$ be a bipartite graph with color classes $A$ and $B$, 
and let $W\in \{A, B\}$. 
\begin{rmenum} 
\item Then, 
for each $H \in \const{G}$, 
$\pargpart{G}{H} = \{ V(H)\cap A, V(H)\cap B\}$. 
For each $H \in \ainconst{G}{W}$, $\pargpart{G}{H} = \{ V(H)\cap W\}$. 
\item 
For any $H_1, H_2\in\ecomp{G}$ with $H_1\neq H_2$, 
$H_1 \yield H_2$ does not hold. 
\end{rmenum}
\end{observation} 

Under Observation~\ref{obs:bi2basilica}, 
we define $\admcomp{G}{W}$ as the set $\{ C\in\dmcomp{G}: \sd{C} \subseteq W\}$. 
Define a mapping $f_W: \const{G}\cup \ainconst{G}{W} \rightarrow  \admcomp{G}{W}$ as $f_W(C) := \ds{V(C)\cap W}$ for $C \in \const{G}$. 
The next statement is obvious from Observation~\ref{obs:bi2basilica}.

\begin{observation} \label{obs:gdm2dm} 
The mapping $f_W$ is a bijection; and, 
 for any $C_1, C_2\in\ecomp{G}$, 
$C_1 \dm{W} C_2$ holds if and only if $f(C_1)\gdm f(C_2)$ holds. 
\end{observation}

According to Theorem~\ref{thm:gdm} and Observation~\ref{obs:gdm2dm}, 
the system $(\const{G}\cup \ainconst{G}{W}, \dm{W})$ is a poset.  
 Observations~\ref{obs:dm2inconst} and \ref{obs:gdm2dm} now prove Theorem~\ref{thm:dm}.

\section{Computational Properties} \label{sec:alg} 

Given a graph $G$,  its basilica decomposition 
can be computed in $O(|V(G)|\cdot |E(G)|)$ time~\cite{DBLP:conf/isaac/Kita12, DBLP:journals/corr/abs-1205-3816}. 
Assume that the basilica decomposition of $G$ is given. 
From the definition of $\gdmo$, 
the poset $(\dmcomp{G}, \gdm)$ can be computed in $O(|\gpart{G}|\cdot |E(G)|)$ time,
and accordingly, in $O(|V(G)|\cdot |E(G)|)$ time. 
According to the definition of $\nono$, 
given the poset $(\dmcomp{G}, \gdm)$, 
the TFR poset $\gdmposet{G}$ can be obtained in $O(|V(G)|)$ time. 
Therefore, the next thereom can be stated. 

\begin{theorem} 
Given a graph $G$, 
the TFR poset $\gdmposet{G}$ can be computed in $O(|V(G)|\cdot |E(G)|)$ time. 
\end{theorem}

\bibliographystyle{splncs03.bst}
\bibliography{gdm.bib}

\begin{thebibliography}{10}
\providecommand{\url}[1]{\texttt{#1}}
\providecommand{\urlprefix}{URL }

\bibitem{duff1986direct}
Duff, I.S., Erisman, A.M., Reid, J.K.: Direct methods for sparse matrices.
  Clarendon press Oxford (1986)

\bibitem{dm1958}
Dulmage, A.L., Mendelsohn, N.S.: Coverings of bipartite graphs. Canadian
  Journal of Mathematics  10(4),  516--534 (1958)

\bibitem{dm1959}
Dulmage, A.L., Mendelsohn, N.S.: A structure theory of bi-partite graphs.
  Trans. Royal Society of Canada. Sec. 3.  53,  1--13 (1959)

\bibitem{dm1963}
Dulmage, A.L., Mendelsohn, N.S.: Two algorithms for bipartite graphs. Journal
  of the Society for Industrial and Applied Mathematics  11(1),  183--194
  (1963)

\bibitem{edmonds1965}
Edmonds, J.: Paths, trees and flowers. Canadian Journal of Mathematics  17,
  449--467 (1965)

\bibitem{gallai1964}
Gallai, T.: Maximale systeme unabh{\"{a}}ngiger kanten. A Magyer
  Tudom{\'{a}}nyos Akad{\'{e}}mia: 
  Int{\'{e}}zet{\'{e}}nek K{\"{o}}zlem{\'{e}}nyei  8,  401--413 (1964)

\bibitem{DBLP:conf/isaac/Kita12}
Kita, N.: {A Partially Ordered Structure and a Generalization of the Canonical
  Partition for General Graphs with Perfect Matchings}. In: Chao, K.M., Hsu,
  T.s., Lee, D.T. (eds.) 23rd Int. Symp. Algorithms Comput. {ISAAC} 2012.
  Lecture Notes in Computer Science, vol. 7676, pp. 85--94. Springer (2012)

\bibitem{DBLP:journals/corr/abs-1205-3816}
Kita, N.: {A Partially Ordered Structure and a Generalization of the Canonical
  Partition for General Graphs with Perfect Matchings}. CoRR  abs/1205.3 (2012)

\bibitem{DBLP:conf/cocoa/Kita13}
Kita, N.: {Disclosing Barriers: {A} Generalization of the Canonical Partition
  Based on Lov{{\'{a}}}sz's Formulation}. In: Widmayer, P., Xu, Y., Zhu, B.
  (eds.) 7th International Conference of Combinatorial Optimization and
  Applications, {COCOA} 2013. Lecture Notes in Computer Science, vol. 8287, pp.
  402--413. Springer (2013)

\bibitem{kita2015graph}
Kita, N.: A graph theoretic proof of the tight cut lemma. arXiv preprint
  arXiv:1512.08870  (2015)

\bibitem{kotzig1959a}
Kotzig, A.: Z te\'orie kone\v{c}n\'ych grafov s line\'arnym faktorom. {I} ({\em
  in slovak}). Mathematica Slovaca  9(2),  73--91 (1959)

\bibitem{kotzig1959b}
Kotzig, A.: Z te\'orie kone\v{c}n\'ych grafov s line\'arnym faktorom. {II}
  ({\em in slovak}). Mathematica Slovaca  9(3),  136--159 (1959)

\bibitem{kotzig1960}
Kotzig, A.: Z te\'orie kone\v{c}n\'ych grafov s line\'arnym faktorom. {III}
  ({\em in slovak}). Mathematica Slovaca  10(4),  205--215 (1960)

\bibitem{lovasz1972structure}
Lov{\'{a}}sz, L.: {On the structure of factorizable graphs}. Acta Math.
  Hungarica  23(1-2),  179--195 (1972)

\bibitem{lovasz2009matching}
Lov{\'a}sz, L., Plummer, M.D.: Matching theory, vol. 367. American Mathematical
  Soc. (2009)

\bibitem{schrijver2003}
Schrijver, A.: Combinatorial optimization: polyhedra and efficiency, vol.~24.
  Springer Science \& Business Media (2002)

\end{thebibliography}

\end{document}